\documentclass[a4paper,reqno,12pt]{amsart}

\usepackage{fullpage}
\usepackage[utf8]{inputenc}
\usepackage[english]{babel}
\usepackage{amssymb}
\usepackage{hyperref}

\newcommand{\so}{\mathfrak{so}}

\DeclareMathOperator{\ad}{ad}
\DeclareMathOperator{\Exp}{Exp}
\DeclareMathOperator{\End}{End}

\theoremstyle{plain}
\newtheorem{theorem}{Theorem}[section]
\newtheorem{lemma}[theorem]{Lemma}
\newtheorem*{utheorem}{Main Theorem}

\theoremstyle{definition}
\newtheorem{example}[theorem]{Example}

\theoremstyle{remark}
\newtheorem{remark}[theorem]{Remark}

\numberwithin{equation}{section}

\begin{document}

\title[The distribution of symmetry]{The distribution of symmetry of a naturally reductive nilpotent Lie group}

\author{Silvio Reggiani}

\address{CONICET and Universidad Nacional de Rosario, Dpto. de Matemática, ECEN-FCEIA, Av.\ Pellegrini 250, 2000 Rosario, Argentina }

\email{reggiani@fceia.unr.edu.ar}

\urladdr{\url{http://www.fceia.unr.edu.ar/~reggiani}}

\date{\today}

\thanks{Supported by CONICET. Partially supported by ANPCyT and SeCyT-UNR}



\begin{abstract}
 We show that the distribution of symmetry of a naturally reductive nilpotent Lie group coincides with the invariant distribution induced by the set of fixed vectors of the isotropy. This extends a known result on compact naturally reductive spaces. We also address the study of the quotient by the foliation of symmetry.
\end{abstract}

\maketitle

\section{Introduction}

The concept of index of symmetry $i_{\mathfrak s}(M)$ of a homogeneous Riemannian manifold $M$ is introduced in \cite{olmos-reggiani-tamaru-2014}. Roughly speaking, $i_{\mathfrak s}(M)$ a geometric invariant which measures how far is $M$ from being a symmetric space. Two objects closely related to the index of symmetry are the distribution and the foliation of symmetry. The distribution of symmetry, whose rank equals the index of symmetry, is an integrable distribution with totally geodesic leaves which is invariant under the full isometry group of $M$. Moreover, its integral manifolds are extrinsically isometric to a global symmetric space and form the foliation of symmetry. When $M$ is compact, there are several general results about the index of symmetry and large families of examples. For instance, $i_{\mathfrak s}(M)$ is known for compact naturally reductive spaces and flag manifolds (see \cite{olmos-reggiani-tamaru-2014,podesta-2015}). Moreover, the article \cite{berndt-olmos-reggiani-2017} contains several structure results and the classification of compact spaces with low co-index of symmetry. The situation in the non compact case is more complicated. To our best knowledge, the only result about non compact spaces appear in \cite{reggiani-2016}, where $i_{\mathfrak s}(M)$ is computed for $3$-dimensional unimodular Lie groups.

The goal of this short note is to extend the results of \cite{olmos-reggiani-tamaru-2014} to a very important family of non compact naturally reductive spaces. More precisely, it is proved in \cite{olmos-reggiani-tamaru-2014} that the distribution of symmetry of a compact (irreducible) naturally reductive space, with respect to an appropriate presentation, coincides with the invariant distribution induced by the set of fixed points of the isotropy group and the leaf of symmetry is always a symmetric space of the group type. We manage to prove the following theorem.

\begin{utheorem}
  Let $N$ be a naturally reductive nilpotent Lie group without Euclidean factor. Then the distribution of symmetry of $N$ coincides with $I(N)$-invariant distribution determined by the fixed points of the isotropy representation $I(N)_p \to T_pN$ at a given point $p \in N$.
\end{utheorem}

The techniques we used to prove this theorem are completely different from the ones used in \cite{olmos-reggiani-tamaru-2014}. Instead we use a beautiful presentation given by J.\ Lauret of $2$-nilpotent Lie groups arising from orthogonal representations of compact Lie algebras \cite{lauret-1999}. Recall that, in contrast with the compact case, the leaf of symmetry is not a symmetric space of the group type but an Euclidean space (see Remark \ref{leaf-of-symmetry}). We also study the quotient by the foliation of symmetry and recover some non trivial vector bundles over Euclidean spaces. 

\section{Preliminaries}
\label{sec:preliminaries}

\subsection{The index of symmetry of a homogeneous manifold}

Let us briefly review the basic notions on the so-called index of symmetry. For more details, see \cite{olmos-reggiani-tamaru-2014}. Let $M$ be a homogeneous Riemannian manifold and let $p \in M$. Let us denote by $\mathfrak K(M)$ the Lie algebra of Killing vector fields on $M$ and consider
\[
  \mathfrak s_p = \{X_p: X \in \mathfrak K(M) \text{ and } (\nabla X)_p = 0\},
\]
where $\nabla$ denotes the Levi-Civita connection of $M$. It is easy to see that $\mathfrak s_p$ is an invariant subspace for the isotropy representation $I(M)_p \to T_pM$, where $I(M)$ is the full isometry group of $M$. Since $I(M)$ is transitive on $M$, we get an $I(M)$-invariant distribution $\mathfrak s$ on $M$, which is called the \emph{distribution of symmetry} of $M$. It is known that $\mathfrak s$ is an autoparallel distribution of $M$ (i.e., it is integrable with totally geodesic leaves). The \emph{index of symmetry} of $M$ is defined to be the rank of $\mathfrak s$. Recall that $M$ is a symmetric space if and only if its index of symmetry equals $\dim M$. Moreover, as shown in \cite{reggiani-2016}, if $M$ is not a symmetric space, then its index of symmetry is at most $\dim M - 2$. Finally, if we denote by $\mathcal L_p$ the leaf of the distribution of symmetry by $p \in M$, then $\mathcal L$ is the so-called \emph{foliation of symmetry} of $M$. Notice that the leaves of the foliation $\mathcal L$ are totally geodesic submanifolds of $M$ which are extrinsically globally symmetric spaces (see \cite{eschenburg-olmos-1994}).

\subsection{Naturally reductive nilpotent Lie groups}

Recall that a \emph{naturally reductive space} is a homogeneous Riemannian space $M = G/H$ such that $\mathfrak g = \mathfrak h \oplus \mathfrak m$ where $\mathfrak g$ and $\mathfrak h$ are the Lie algebras of $G$ and $H$ respectively, $\mathfrak m$ is a subspace of $\mathfrak g$ such that $[\mathfrak h, \mathfrak m] \subset \mathfrak m$ and, for all $X \in \mathfrak m$, $[X, -\,]_{\mathfrak m} \in \so(\mathfrak m)$, where the inner product on $\mathfrak m$ comes from the natural identification with $T_pM$. Geometrically, this means that $\Exp(tX) \cdot p$ is a geodesic through $p$ for all $X \in \mathfrak m$.

We are interested in \emph{naturally reductive nilpotent Lie groups} which are by definition nilpotent Lie groups endowed with a naturally reductive left-invariant metric (with respect to the full isometry group). In \cite{gordon-1985}, C. Gordon proved that a naturally reductive nilpotent Lie group is at most $2$-step nilpotent. Moreover, by a result due to J.\ Lauret, the $2$-step nilpotent naturally reductive Lie groups without Euclidean factor are the ones constructed in the following way. Let $\mathbb V$ be an Euclidean vector space and $\mathfrak g$ be a compact Lie algebra. Let $\pi: \mathfrak g \to \so(\mathbb V)$ be an orthogonal representation of $\mathfrak g$. Assume further that $\pi$ is faithful and without trivial subrepresentations, i.e., $\bigcap_{x \in \mathfrak g}\ker\pi(x) = 0$. Let us consider an $\ad$-invariant inner product $\langle-,-\rangle_{\mathfrak g}$ on $\mathfrak g$ and denote by $\langle-,-\rangle_{\mathbb V}$ the inner product on $\mathbb V$. We define on $\mathfrak n = \mathfrak g \oplus \mathbb V$ the Lie bracket $[-,-]_{\mathfrak n}$ given by
\begin{itemize}
  \item $[\mathfrak g, \mathfrak g]_{\mathfrak n}  = [\mathfrak g, \mathbb V] = 0$;
  \item $[\mathbb V, \mathbb V]_{\mathfrak n}  \subset \mathfrak g$;
  \item $\langle[v, w], x\rangle_{\mathfrak g} = \langle\pi(x)v, w\rangle_{\mathbb V}$ for all $x \in \mathfrak g$, $v, w \in \mathbb V$.
\end{itemize}

Clearly, $\mathfrak n$ is a $2$-step nilpotent Lie algebra. Let $N$ be the $2$-step nilpotent Lie group with Lie algebra $N$, and let us consider on $N$ the left-invariant metric induced by the inner product $\langle-,-\rangle = \langle-,-\rangle_{\mathfrak g} \oplus \langle-,-\rangle_{\mathbb V}$.

\begin{theorem}
  [\cite{lauret-1999}]
  $(N, \langle-,-\rangle)$ is a $2$-step nilpotent naturally reductive Lie group. Reciprocally, every naturally reductive nilpotent Lie group without Euclidean factor arises in this way.
\end{theorem}

The connected component of the full isometry group of $N$ can be computed as follows. It is well known that the Lie algebra $\mathfrak i(N)$ of the full isometry group $I(N)$ is given by $\mathfrak k \ltimes \mathfrak n$, where the Killing fields identified with $\mathfrak n$ are the right-invariant fields on $N$, and the Lie algebra $\mathfrak k$ of the isotropy group consists of the orthogonal derivations of $\mathfrak n$. A complete description of $\mathfrak k$ is also given in \cite{lauret-1999}: let us write $\mathfrak g = \bar{\mathfrak g} \oplus \mathfrak c$, where $\mathfrak c$ is the center of $\mathfrak g$ and $[\bar{\mathfrak g}, \bar{\mathfrak g}] = \bar{\mathfrak g}$ is semisimple. Then $\mathfrak k = \bar{\mathfrak g} \oplus \mathfrak u$, where
\[
  \mathfrak u = \mathfrak \End_\pi(\mathbb V) \cap \so(\mathbb V) = \{A \in \so(\mathbb V): [A, \pi(x)] = 0 \text{ for all } x \in \mathfrak g\}
\]
are the orthogonal intertwining operators for the representation $\pi$. Recall that the restriction of the (derived) isotropy representation to $\bar{\mathfrak g}$ is given by $(\ad x, \pi(x))$ for all $x \in \bar{\mathfrak g}$.

We finally note that, according to \cite[Lemma 3.11]{lauret-1999}, if one writes $\mathbb V = \mathbb V_1 \oplus \cdots \oplus \mathbb V_k$ as an orthogonal sum of irreducible subrepresentations of $\pi$, then for each $1 \le i \le k$ there exists a skew-symmetric operator $J_i: \mathbb V_i \to \mathbb V_i$ such that $J_i^2 = -I$ and, for each $h \in \mathfrak c$,
\begin{equation}
  \label{eq:2}
  \pi(h) = \lambda_i(h)J_i
\end{equation}
for some $\lambda_i(h) \in \mathbb R$.

Covariant derivatives of left-invariant vector fields on $N$ are easy to compute (see for instance \cite{eberlein-1994}). Given $X \in \mathfrak n$, we denote by $X^*$ the associated Killing right-invariant field. The following result is straightforward.

\begin{lemma}\label{sec:natur-reduct-nilp}
  Consider $\mathfrak n = \mathfrak g \oplus \mathbb V$ as before and let $X, Y \in \mathbb V$, $Z, T \in \mathfrak g$. Then
  \begin{enumerate}
  \item $(\nabla_{X^*}Y^*)_e = -\frac12[X^*, Y^*]_e$;
  \item $(\nabla_{X^*}Z^*)_e = (\nabla_{Z^*}X^*)_e =  -\frac12(\pi(Z)X)_e$;
  \item $(\nabla_{Z^*}T^*)_e = 0$.
  \end{enumerate}
\end{lemma}

\section{Proof of the Main Theorem}

Let us consider and arbitrary Killing field $\tilde Y$ on $N$, and write $\tilde Y = Y^* + Y^{\mathfrak k}$, where $Y^*$ is a right-invariant Killing field and $X^{\mathfrak k}$ is an element of the isotropy subalgebra, i.e., $(Y^{\mathfrak k})_e = 0$. We again decompose $Y^*$ according to the decomposition $\mathfrak n = \mathfrak g \oplus \mathbb V = (\bar{\mathfrak g} \oplus \mathfrak c) \oplus \mathbb V$, and $Y^{\mathfrak k}$ according to $\mathfrak k = \bar{\mathfrak g} \oplus \mathfrak u$. So
\begin{equation}
  \label{eq:1}
  \tilde Y = Y^*_{\mathfrak g} + Y^*_{\mathbb V} + Y^{\mathfrak k} = Y^*_{\bar{\mathfrak g}} + Y^*_{\mathfrak c} + Y^*_{\mathbb V} + Y^{\mathfrak k}_{\bar{\mathfrak g}} + Y^{\mathfrak k}_{\mathfrak u}.
\end{equation}

Let us study the covariant derivative of $\tilde Y$ at the identity.

\begin{lemma}
  If $(\nabla\tilde Y)_e = 0$, then $Y^*_{\bar{\mathfrak g}} = Y^*_{\mathbb V} =  Y^{\mathfrak k}_{\bar{\mathfrak g}} = 0$.
\end{lemma}

\begin{proof}
  Let $X_{\mathfrak g} \in \mathfrak g$ be an central element inside $\mathfrak n$, and denote by $X^*_{\mathfrak g}$ the corresponding right-invariant vector field (which of course coincides with $X_{\mathfrak g}$). Using Lemma~\ref{sec:natur-reduct-nilp}, we have that
  \begin{align*}
    0 & = (\nabla_{X^*_{\mathfrak g}}\tilde Y)_e \\
      & = (\nabla_{X^*_{\mathfrak g}}Y^*_{\mathbb V})_e + (\nabla_{X^*_{\mathfrak g}}Y^{\mathfrak k}_{\bar{\mathfrak g}})_e \\
      & = -\frac12\pi(X_{\mathfrak g})Y_{\mathbb V}|_e + [Y^{\mathfrak k}_{\bar{\mathfrak g}}, {X^*_{\mathfrak g}}]_e,
  \end{align*}
  where the bracket $[Y^{\mathfrak k}_{\bar{\mathfrak g}}, {X^*_{\mathfrak g}}]$ is took in $\mathfrak g$. So, both summands in the last equality belong to two different orthogonal subspaces of $T_eN$. Moreover, since $X^*_{\mathfrak g}$ is arbitrary and $\bar{\mathfrak g}$ is semisimple, we conclude that $Y^*_{\mathbb V} = Y^{\mathfrak k}_{\bar{\mathfrak g}} = 0$. Recall that, now $\tilde Y$ takes the form
  \begin{equation*}
    \label{eq:1bis}
    \tilde Y = Y^*_{\mathfrak g} + Y^{\mathfrak k}_{\mathfrak u} = Y^*_{\bar{\mathfrak g}} + Y^*_{\mathfrak c} + Y^{\mathfrak k}_{\mathfrak u}.
  \end{equation*}

  Let $\mathbb W$ be a $\pi$-irreducible subspace of $\mathbb V$. So, from (\ref{eq:2}), we have that $\pi(Y_{\mathfrak c}) = \lambda J$, for some $\lambda \in \mathbb R$ and $J \in \so(\mathbb W)$ is such that $J^2 = -I$. Notice that, since $\pi$ is faithful, $Y^*_{\bar{\mathfrak g}} \notin \mathbb RJ$, and since $u$ commutes with every element in $\pi(\mathfrak g)$, $Y^*_{\bar{\mathfrak g}} \notin \mathbb RY^{\mathfrak k}_{\mathfrak u}$. Now from Lemma \ref{sec:natur-reduct-nilp} again, for each $w \in \mathbb W$ one gets
  \[
    0 = (\nabla_w\tilde Y)_e = -\frac12 \pi(Y_{\bar{\mathfrak g}})w - \frac{\lambda}{2}Jw + Y^{\mathfrak k}_{\mathfrak u}w.
  \]
  This implies that $Y^*_{\bar{\mathfrak g}} = 0$, since otherwise $Y^{\mathfrak k}_{\mathfrak u}$ would not commute with $\pi(\bar{\mathfrak g})$.
\end{proof}

Keep the notation from the proof of the above lemma. Assuming that $(\nabla \tilde Y)_e = 0$ one gets
\begin{equation}
  \label{eq:3}
  \tilde Y = Y^*_{\mathfrak c} + Y^{\mathfrak k}_{\mathfrak u}
\end{equation}
and when restricted to $\mathbb W$, one has $Y^{\mathfrak k}_{\mathfrak u} = \frac{\lambda}{2}J$. In other words, every element in $\mathfrak c$ is in the direction of the distribution of symmetry through the identity.

\begin{remark}
  Recall that $\mathfrak c$ regarded as a subspace of $T_eN$ is the subspace where the isotropy algebra acts trivially. So, the distribution of symmetry of $N$ is the $I(N)$-invariant distribution induced for the fixed points of the full isotropy representation.
\end{remark}

\begin{example}
 The Heisenberg Lie group $H_n$ can be obtained from the representation $\pi: \mathbb R \to \so(2n)$, where $\pi$ is induced by the multiplication by $\sqrt{-1}$ in $\mathbb R^{2n} \simeq \mathbb C^n$. So, the metric from Section \ref{sec:preliminaries} for this particular representation is the so-called type-$H$ metric, which has index of symmetry $1$.  We want to mention the work Agricola, Ferreira and Friedrich \cite{agricola-ferreira-friedrich-2014} where they classify all the naturally reductive spaces up to dimension $6$. In that article the naturally reductive left-invariant metric in the $5$-dimensional Heisenberg group $H_2$, are attached to the representations $\pi: \mathbb R \to \so(2n)$ where $\pi$ acts as a different multiple of the multiplication by $\sqrt{-1}$ in each copy of $\mathbb C$ inside $\mathbb R^{2n} \simeq \mathbb C^n$. According to our main theorem, all these metrics also have index of symmetry equal to $1$.
\end{example}

\begin{remark}
  \label{leaf-of-symmetry}
  Let $\mathcal L$ be the foliation of symmetry of $N$ (i.e., the one given by the leaves of the distribution of symmetry). With a similar argument as in \cite{reggiani-2016}, one has that the quotient map $N \to N/\mathcal L$ is a Riemannian submersion. Here, $N/\mathcal L$ is isometric to the naturally reductive nilmanifold obtained by the restriction of the representation $\pi|_{\bar{\mathfrak g}}: \bar{\mathfrak g} \to \so(\mathbb V)$, where we ignore the center $\mathfrak c$ of $\mathfrak g$. Note that, in contrast with most of the known examples, the quotients by the leaves of symmetry is not a symmetric space. The leave of symmetry is isometric to the Euclidean space of dimension $\dim\mathfrak c$. We also mention that, when $\mathfrak g = \mathfrak c$ is abelian, the quotient $N \to N/\mathcal L$ can be regarded as a non trivial vector bundles over the Euclidean space $\mathbb V$ with fiber $\mathfrak c$.
\end{remark}

\bibliography{/home/silvio/Dropbox/math/bibtex/mybib}
\bibliographystyle{amsalpha}

\end{document}